\documentclass{amsart}
\usepackage{mathptmx}
\usepackage{pxfonts}
\usepackage{amssymb}
\usepackage{amsfonts}
\usepackage{amsmath}
\usepackage{graphicx}
\usepackage{shadow}
\usepackage{color}
\usepackage[all]{xy}
\usepackage[pagebackref]{hyperref}
\newtheorem{thm}{Theorem}[section]
\newtheorem{corollary}[thm]{Corollary}
\newtheorem{lemma}[thm]{Lemma}
\newtheorem{proposition}[thm]{Proposition}

\theoremstyle{definition}

\newtheorem{example}[thm]{Example}

\theoremstyle{remark}
\newtheorem{remark}[thm]{Remark}

\DeclareMathOperator{\Ze}{Z}
\DeclareMathOperator{\J}{Jac}

\DeclareMathOperator{\id}{id}

\DeclareMathOperator{\Ker}{Ker}

\DeclareMathOperator{\Spec}{Spec}

\DeclareMathOperator{\m}{\frak{m}}
\DeclareMathOperator{\nil}{Nil}
\newcommand{\field}[1]{\mathbb{#1}}

\newcommand{\Q }{\field{Q}}
\newcommand{\Z }{\field{Z}}

\begin{document}

\title[Bi-Amalgamated algebras along  ideals]{Bi-Amalgamated algebras along  ideals $^{(\star)}$}
\thanks{$^{(\star)}$ Supported by King Fahd University of Petroleum \& Minerals  under Research Grant \# RG1310.}

\author[S. Kabbaj]{S. Kabbaj $^{(1)}$}\thanks{$^{(1)}$ Corresponding author.}
\address{Department of Mathematics and Statistics, King Fahd University of Petroleum \& Minerals, Dhahran 31261, KSA}
\email{kabbaj@kfupm.edu.sa}

\author[K. Louartiti]{K. Louartiti}
\address{Department of Mathematics, Faculty of Science and Technology, Box 2202, University S. M. Ben Abdellah, Fez, Morocco.}
\email{lokha2000@hotmail.com}

\author[M. Tamekkante]{M. Tamekkante}
\address{Laboratory of Mathematics, Computing and Application, Department of Mathematics, Faculty
of Science, Box 1014, University Mohammed V--Agdal, Rabat, Morocco.}
\email{tamekkante@yahoo.fr}

\date{\today}


\begin{abstract}
Let $f: A\rightarrow B$ and $g: A\rightarrow C$ be two commutative ring homomorphisms and let $J$ and $J'$ be
two ideals of $B$ and $C$, respectively, such that $f^{-1}(J)=g^{-1}(J')$. The \emph{bi-amalgamation} of $A$ with $(B, C)$ along $(J, J')$ with respect to $(f,g)$ is the subring of $B\times C$ given by
$$A\bowtie^{f,g}(J,J'):=\big\{(f(a)+j,g(a)+j') \mid a\in A, (j,j')\in J\times J'\big\}.$$
 This paper investigates ring-theoretic properties of \emph{bi-amalgamations} and capitalizes on previous works carried on various settings of pullbacks and amalgamations. In the second and third sections, we provide examples of bi-amalgamations and show how these constructions arise as pullbacks. The fourth section investigates the transfer of some basic ring theoretic properties to bi-amalgamations and the fifth section is devoted to the prime ideal structure of these constructions.  All new results agree with recent studies in the literature on D'Anna-Finocchiaro-Fontana's amalgamations and duplications.
\end{abstract}

\maketitle

\section{Introduction}

\noindent Throughout, all rings considered are commutative with unity and all modules are unital. The following diagram of ring homomorphisms
$$\xymatrix{ R \ar[d]^{\mu_2}\ar[r]^{\iota_2} & \ar[d]^{\mu_1}T\\A \ar[r]^{\iota_1}&  B}$$
is called  the pullback (or fiber product) of $\mu_1$ and $\iota_1$ if the homomorphism $\iota_2\times \mu_2: R\rightarrow T\times A$, $r\mapsto (\iota_2(r),\mu_2(r))$ induces an isomorphism of $R$ onto the subring of $T\times A$ given by $$\mu_1\times_{B}\iota_1:=\big\{(t,a)\mid \mu_1(t)=\iota_1(a)\big\}.$$
If $\mu_1$ is surjective and $\iota_1$ is injective, the above diagram is called  a conductor square. In this setting, $\iota_2$ and $\mu_2$ are injective and surjective, respectively, and    $\Ker(\mu_1)\cong\Ker(\mu_2)$. By abuse of notation, we view $R$ as a subring of $T$ making $\Ker(\mu_1)=\Ker(\mu_2)$ the largest common ideal of $R$ and $T$; it is called the conductor of $T$ into $R$.

\emph{Amalgamated algebras} are rings which arise as special pullbacks. Their introduction in 2007 by D'Anna and Fontana \cite{DF1,DF2} was motivated by a construction of D. D. Anderson \cite{An} related to a classical construction due to Dorroh \cite{Do} on endowing a ring (without unity) with a unity. The interest of these amalgamations resides, partly, in their ability to cover several basic constructions in commutative algebra, including classical pullbacks (e.g., $D+M$, $A+XB[X]$, $A+XB[[X]]$, etc.), Nagata's idealizations \cite{Hu,N} (also called trivial ring extensions which have been widely studied in the literature), and Boisen-Sheldon's CPI-extensions \cite{BSh2}. The following paragraphs collect background and main contributions on amalgamations.

Let $A$ be a ring, $I$ an ideal of $A$, and $\pi:A\rightarrow \frac{A}{I}$ the canonical surjection. The amalgamated duplication of $A$ along $I$, denoted by $A\bowtie I$, is the special pullback  of $\pi$ and $\pi$; i.e., the subring  of $A \times A$ given by $$A\bowtie I:=\pi\times_{\frac{A}{I}}\pi=\big\{(a,a+i)\mid a\in A, i\in I\big\}.$$
If $I^{2} = 0$, then $A\bowtie I$ coincides with Nagata's idealization $A\ltimes I$.

In 2007, the construction $A\bowtie I$ was introduced and its basic properties were studied by D'Anna and Fontana in \cite{DF1,DF2}. In the firs paper \cite{DF1}, they discussed the main properties of the amalgamated duplication in relation with pullback constructions and special attention was devoted to its ideal-theoretic properties as well as to the topological structure of its prime spectrum. In the second paper \cite{DF2}, they restricted their attention to the case where $I$ is a multiplicative canonical ideal of $A$, that is, $I$ is regular and every regular fractional ideal $J$ of $R$ is $I$-reflexive (i.e., $J=(I :(I :J))$). In particular, they examined contexts where every regular fractional ideal of $A\bowtie I$ is divisorial. Later in the same year, the amalgamated duplication was investigated by D'Anna in \cite{D} with the aim of applying it to curve singularities (over algebraic closed fields) where he proved that the amalgamated duplication of an algebroid curve along a regular canonical ideal yields a Gorenstein algebroid curve \cite[Theorem 14 and Corollary 17]{D}. In 2008, Maimani and Yassemi studied in \cite{MY} the diameter and girth of the zero-divisor graph of an amalgamated duplication. In 2010, Shapiro \cite{Sh} corrected Proposition 3 in \cite{D} and proved a pertinent result asserting that if $A$ is a one-dimensional reduced local Cohen-Macaulay ring and $A\bowtie I$ is Gorenstein, then $I$ must be regular. In 2012, in \cite{CJKM}, the authors established necessary and sufficient conditions for an amalgamated duplication of a ring along an ideal to inherit Pr\"ufer conditions (which extend the notion of Pr\"ufer domain to commutative rings with zero divisors). The new results yielded original and new families of examples issued from amalgamated duplications subject to various Pr\"ufer conditions.

In 2009 and 2010, D'Anna, Finocchiaro, and Fontana considered the more general context of amalgamated algebra $$A\bowtie^{f} J:=\big\{(a,f(a)+j)\mid a\in A, j\in J\big\}$$ for a given homomorphism of rings $f: A\rightarrow B$ and ideal $J$ of $B$. In particular, they have studied these amalgamations in the frame of pullbacks which allowed them to establish numerous (prime) ideal and ring-theoretic basic properties for this new construction. In \cite{DFF1}, they provided necessary and sufficient conditions for $A\bowtie^{f} J$ to inherit the notions of Noetherian ring, domain, and reduced ring and characterized pullbacks that can be expressed as amalgamations. In \cite{DFF2}, they provided a complete description of the prime spectrum of $A\bowtie^{f} J$ and gave bounds for its Krull dimension.

Let $\alpha: A\rightarrow C$, $\beta: B\rightarrow C$ and $f:A\rightarrow B$ be ring homomorphisms.  In the aforementioned papers \cite{DFF1,DFF2}, the authors studied amalgamated algebras within the frame of pullbacks $\alpha\times \beta$ such that $\alpha=\beta\circ f$ \cite[Propositions 4.2 and 4.4]{DFF1}. In this work, we are interested in new constructions, called \emph{bi-amalgamated algebras} (or \emph{bi-amalgamations}), which arise as pullbacks $\alpha \times \beta$ such that the following diagram of ring homomorphisms
$$\xymatrix{ A \ar[d]^{g}\ar[r]^{f} & \ar[d]^{\alpha}B\\ C \ar[r]^{\beta}&  D}$$
is commutative with $\alpha\circ \pi_B(\alpha\times \beta)=\alpha\circ f(A)$, where $\pi_B$ denotes the canonical projection of $B\times C$ over $B$. Namely, let $f: A\rightarrow B$ and $g: A\rightarrow C$ be two ring homomorphisms and let $J$ and $J'$ be
two ideals of $B$ and $C$, respectively, such that $f^{-1}(J)=g^{-1}(J')$. The \emph{bi-amalgamation of $A$ with $(B, C)$ along $(J, J')$ with respect to $(f,g)$} is the subring of $B\times C$ given by
$$A\bowtie^{f,g}(J,J'):=\big\{(f(a)+j,g(a)+j') \mid a\in A, (j,j')\in J\times J'\big\}.$$

 This paper investigates ring-theoretic properties of \emph{bi-amalgamations} and capitalizes on previous works carried on various settings of pullbacks and amalgamations. In the second and third sections, we provide examples of bi-amalgamations and show how these constructions arise as pullbacks. The fourth section investigates the transfer of some basic ring theoretic properties to bi-amalgamations and the fifth section is devoted to the prime ideal structure of these constructions.  All new results agree with recent studies in the literature on D'Anna-Finocchiaro-Fontana's amalgamations and duplications.

 Throughout, for a ring $R$, $Q(R)$ will denote the total ring of quotients and $\Ze(R)$ and $\J(R)$ will denote, respectively, the set of zero divisors and Jacobson radical of $R$. Finally, $\Spec(R)$ shall denote the set of prime ideals of $R$.

\section{Examples of bi-amalgamations}

Notice, first, that every amalgamated duplication is an amalgamated algebra and every amalgamated algebra is a bi-amalgamated algebra, as seen below.

\begin{example}[The amalgamated algebra]\label{aa}  Let $f:A\rightarrow B$ be a ring homomorphism and $J$ an ideal of $B$.
Set $I:=f^{-1}(J)$ and $\iota:=\id_A$. Thus,
\begin{eqnarray*}
  A\bowtie^{\iota, f}(I,J) &=& \big\{(a+i,f(a)+j) \mid  a\in A , (i,j)\in I\times J\big\} \\
   &=&\big\{(a+i,f(a+i)+j-f(i))\mid  a\in A , (i,j)\in I\times J\big\} \\
   &=& \big\{(a,f(a)+j)\mid  a\in A , j\in J\big\} \\
   &=& A\bowtie^fJ.
\end{eqnarray*}
\end{example}

Further, the subring $f(A)+J$ of $B$ can be regarded as a bi-amalgamation; precisely:

\begin{remark}\label{f(A)+J}
Let $f:A\rightarrow B$ be a ring homomorphism and $J$ and ideal of $B$. Set $I:=f^{-1}(J)$ and consider the canonical projection $\pi: A\rightarrow A/I$. Then, one can easily check that
\begin{eqnarray*}
f(A)+J      &\cong  &\big\{(\bar{a},f(a)+j)\mid  a\in A , j\in J\big\} \\
            &=      &A\bowtie^{\pi,f}(0,J).
\end{eqnarray*}
\end{remark}

 In particular, Boisen-Sheldon's CPI-extensions \cite{BSh2} can also be viewed as bi-amalgamations.

\begin{example}[The CPI-extension]
Let $A$ be a ring and let $I$ be an ideal of $A$. Then $\overline{S}:=(A/I)\setminus\Ze(A/I)$ and $S:=\{s\in A \mid \bar{s}\in \overline{S}\}$ are multiplicatively closed subsets of $A/I$ and $A$, respectively. Let $\varphi: S^{-1}A\rightarrow Q(A/I)=(\overline{S})^{-1}(A/I)$ and $f: A\rightarrow S^{-1}A$ be the canonical ring homomorphisms. Then, the subring $$C(A,I):=\varphi^{-1}(A/I)=f(A) + S^{-1}I$$ of $S^{-1}A$ is called the CPI-extension of $A$ with respect to $I$ (in the sense of Boisen-Sheldon). Now, let $\pi: A\rightarrow A/I$ be the canonical projection. From Remark~\ref{f(A)+J}, we have
$$A\bowtie^{\pi,f}(0, S^{-1}I)\cong f(A) + S^{-1}I=C(A,I).$$
\end{example}

Other known families of rings stem from Remark~\ref{f(A)+J}; namely, those issued from extensions of rings $A\subset B$ (including classic pullbacks).

\begin{example}[The ring $A+J$]
Let $i:A\hookrightarrow B$ be an embedding of rings, $J$ and ideal of $B$, $I:=A\cap J$, and  $\pi: A\rightarrow A/I$ the canonical projection. From Remark~\ref{f(A)+J}, the subring  $A+J$ of $B$ can arise as a bi-amalgamation via $$A+J\cong A\bowtie^{\pi,i}(0,J)$$
and, consequently, so do most classic pullback constructions such as $A+XB[X]$ (via $A\subset B[X]$ and $XB[X]$), $A+XB[[X]]$  (via $A\subset B[[X]]$ and $XB[[X]]$), and $D+M$ (via $D\subset T$ and $M$ ideal of $T$ with $D\cap M=0$).
\end{example}

In the next section, as an application of Proposition~\ref{cara}, we will see that some
glueings of prime ideals \cite{Ped,Tam,Tra,Yan} can be viewed as bi-amalgamations. We
close this section with an explicit (non-classic pullback) example; namely, the ring
$R:=\Z[X]+(X^{2}+1)\Q[X]$ which lies between $\Z[X]$ and $\Q[X]$.

\begin{example}
Let $i:\Z[X]\hookrightarrow \Q[X]$ be the natural embedding and consider the ring homomorphism $\pi:\Z[X]\rightarrow \Z[i]$, $p(X)\mapsto p(i)$. Clearly, $(X^2+1)\Q[X]\cap \Z[X]=(X^2+1)$ and $\dfrac{\Z[X]}{(X^2+1)}\cong\Z[i]$ so that $$R:=\Z[X]+(X^{2}+1)\Q[X]\cong \Z[X]\bowtie^{\pi,i}\big(0,(X^2+1)\Q[X]\big).$$
\end{example}

\section{Pullbacks and bi-amalgamations}

\noindent Throughout, let $f: A\rightarrow B$ and $g: A\rightarrow C$ be two ring homomorphisms and $J,J'$ two ideals of $B$ and $C$, respectively, such that $I:=f^{-1}(J)=g^{-1}(J')$. Let $A\bowtie^{f,g}(J,J')$ denote the bi-amalgamation of $A$ with $(B, C)$ along $(J, J')$ with respect to $(f,g)$.

This section sheds light on the correlation between pullback constructions and bi-amalgamations. We first show how every bi-amalgamation can arise as a natural pullback.

\begin{proposition}\label{prop1}
Consider the ring homomorphisms  $\alpha: f(A)+J\rightarrow A/I$, $f(a)+j\mapsto \bar{a}$ and $\beta: g(A)+J'\rightarrow A/I$, $g(a)+j'\mapsto \bar{a}$. Then, the bi-amalgamation is determined by the following pullback
$$\xymatrix{
A\bowtie^{f,g}(J,J')    \ar@{>>}[d]  \ar@{>>}[r]    & \ar[d]^{\alpha}f(A)+J\\
g(A)+J'                 \ar[r]^{\beta}              &  A/I
}$$
that is
$$A\bowtie^{f,g}(J,J')=\alpha\times_{\frac{A}{I}}\beta.$$
\end{proposition}

\begin{proof}
Note that the mappings $\alpha$ and $\beta$ are well defined since $I:=f^{-1}(J)=g^{-1}(J')$ and are ring homomorphisms. Further, the inclusion $A\bowtie^{f,g}(J,J')\subseteq \alpha\times \beta $ is trivial. On the other hand,
$$\alpha\times_{\frac{A}{I}}\beta=\big\{(f(a)+j,g(b)+j')\mid a,b\in A,\ (j,j')\in J\times J',\ \alpha(a)=\beta(b)\big\}.$$
The condition $\alpha(a)=\beta(b)$ means that $f(b-a)\in J$ and $g(b-a)\in J'$. It follows that $g(b)+j'=g(a)+(j'+g(b-a))$ with $j'+g(b-a)\in J'$. Therefore, $\alpha\times \beta \subseteq A\bowtie^{f,g}(J,J')$.
\end{proof}

 Next, we see how bi-amalgamations  can be represented as conductor squares.

\begin{proposition}\label{conductor}
Consider the following ring homomorphisms
$$\begin{array}{cccc}
  \iota_1:  & \dfrac{A}{I}  & \longrightarrow   & \displaystyle\frac{f(A)+J}{J}\times \frac{g(A)+J'}{J'} \\
            & \bar{a}       & \longmapsto       &  \left(\overline{f(a)},\overline{g(a)}\right)\\
            &&&\\
 \mu_2:     & A\bowtie^{f,g}(J,J')  & \longrightarrow   & \dfrac{A}{I} \\
            & (f(a)+j,g(a)+j')      & \longmapsto       &  \bar{a}
\end{array}$$
Then, the following diagram
$$\xymatrix@R=1.5cm  @C=1.5cm{ A\bowtie^{f,g}(J,J') \ar@{>>}[d]^{\mu_2}\ar[r]^{\iota_2} & \ar@{>>}[d]^{\mu_1}(f(A)+J)\times (g(A)+J')\\
\displaystyle\frac{A}{I}\ar[r]^{\iota_1}&  \displaystyle \frac{f(A)+J}{J}\times \frac{g(A)+J'}{J'}
}$$
is a conductor square with conductor $\Ker(\mu_1)=J\times J'$, where $\iota_2$ is the natural embedding and $\mu_1$ is the canonical surjection.
\end{proposition}

\begin{proof}
The mappings $\iota_1$ and $\mu_2$ are well defined since $I=f^{-1}(J)=g^{-1}(J')$ and are ring homomorphisms. Next,
set $R:=\mu_1\times \iota_1$ and let $a\in A$ and $(j,j')\in J\times J'$. Then
$$\iota_2\times \mu_2\big((f(a)+j,g(a)+j')\big)=\big((f(a)+j,g(a)+j'),\bar{a}\big)$$
with
$$\mu_1\big((f(a)+j,g(a)+j')\big)=\big(\overline{f(a)},\overline{g(a)}\big)=\iota_1(\bar{a}).$$
Thus, $\iota_2\times \mu_2\big(A\bowtie^{f,g}(J,J')\big)\subseteq R$. Now, let $\big((f(a)+j,g(a')+j'),\bar{b}\big) \in R$. Then
$$\big(\overline{f(a)},\overline{g(a')}\big)=\big(\overline{f(b)},\overline{g(b)}\big).$$
Hence, $f(a-b)\in J$ and $g(a'-b)\in J'$. Whence,
$$\iota_2\times \mu_2\big((f(b)+f(a-b)+j,g(b)+g(a'-b)+j')\big)=\big((f(a)+j,g(a')+j'),\bar{b}\big).$$
It follows that $\iota_2\times \mu_2$ induces an isomorphism of $A\bowtie^{f,g}(J,J')$ onto $R$ since $\iota_2\times \mu_2$ is injective. Consequently, the above diagram is a pullback. Moreover, it is clear that $\iota_1$ is injective and that $\Ker(\mu_1)=J\times J'=\Ker(\mu_2)$.
\end{proof}

The next result characterizes pullbacks that can arise as bi-amalgamations.

\begin{proposition}\label{cara}
Consider the following diagram
$$\xymatrix{ A \ar[d]^{g}\ar[r]^{f} & \ar[d]^{\alpha}B\\ C \ar[r]^{\beta}&  D}$$
of ring homomorphisms and let $\pi:B\times C\rightarrow B$  be the canonical projection. Then, the following conditions are equivalent:
\begin{enumerate}
\item $\alpha\times_D \beta=A\bowtie^{f,g}(J,J')$, for some ideals $J$ of $B$ and $J'$ of $C$ with $f^{-1}(J)=g^{-1}(J')$;
\item The above diagram is commutative with $\alpha\circ \pi(\alpha\times_D \beta)=\alpha\circ f(A)$.
\end{enumerate}
\end{proposition}

\begin{proof}
$(1)\Rightarrow(2)$ Let $a\in A$. By hypothesis, $(f(a),g(a))\in \alpha\times_D \beta$ so that $\alpha\circ f(a)=\beta\circ g(a)$. Also, we have $\pi(\alpha\times_D \beta)=f(A)+J$. Further, for any $j\in J$, the fact $(j,0)\in A\bowtie^{f,g}(J,J')$ yields $\alpha(j)=\beta(0)=0$. Therefore, $\alpha\circ \pi(\alpha\times_D \beta)=\alpha\circ f(A)$, as desired.

$(2)\Rightarrow(1)$ Let $J:=\Ker(\alpha)$ and $J':=\Ker(\beta)$.  By assumption, for each  $x\in f^{-1}(J)$, $\beta\circ g(x)=\alpha\circ f(x)=0$. Then, $g(x)\in J'$ and hence $f^{-1}(J)\subseteq g^{-1}(J')$. Likewise for the reverse inclusion. Hence $f^{-1}(J)=g^{-1}(J')$. Next, let $(f(a)+j,g(a)+j')\in A\bowtie^{f,g}(J,J')$. We have
$$\alpha(f(a)+j)=\alpha\circ f(a)=\beta\circ g(a)=\beta(g(a)+j')$$
so that $A\bowtie^{f,g}(J,J')\subseteq\alpha\times_D \beta$. On the other hand, let $(b,c)\in \alpha\times_D \beta$. By assumption, there exists $a\in A$ such that
$$\alpha(b)=\alpha\circ \pi(b,c)=\alpha(f(a)).$$
Then, $b-f(a)\in J$. Moreover, we have
$$\beta(c)=\alpha(b)=\alpha(f(a))=\beta(g(a)).$$
Then, $c-g(a)\in J'$. It follows that
$$(b,c)=(f(a)+b-f(a),g(a)+c-g(a)\in A\bowtie^{f,g}(J,J').$$
Consequently, $\alpha\times_D \beta=A\bowtie^{f,g}(J,J')$, completing the proof of the proposition.
\end{proof}

In view of Example~\ref{aa}, Proposition~\ref{cara} recovers the special case of amalgamated algebras, as recorded in the next corollary.

\begin{corollary}[{\cite[Proposition 4.4]{DFF1}}]
Let $\alpha: A\rightarrow D$ and $\beta: B\rightarrow D$ be two ring homomorphisms.
Then, $\alpha\times_D \beta=A\bowtie^{f}J$, for some ideal $J$ of $B$ if and only if $\alpha=\beta\circ f$.
\qed\end{corollary}

We close this section with a brief discussion on Traverso's Glueings of prime
ideals \cite{Ped,Tam,Tra} which are special pullbacks \cite[Lemma 2]{Yan}. So, they can also be viewed as special bi-amalgamations if they satisfy Condition (2) of Proposition~\ref{cara}. Precisely, from \cite[Lemma 1]{Yan}, let $A$ be a Noetherian ring and $B$ an overring of $A$ such that $B$ is a finite $A$-module. Let $p\in\Spec(A)$ and let $p_{1}, ..., p_{n}$ be the prime ideals of $B$ lying over $p$. For each $i$, $\frac{A_{p}}{pA_{p}}$ is a subfield of $\frac{B_{p_{i}}}{p_{i}B_{p_{i}}}$, and let $\overline{\frac{b}{t}}^{i}$ denote the class of the element $\frac{b}{t}$ of $B_{p_{i}}$ modulo $p_{i}B_{p_{i}}$. The ring $A'$ obtained from $B$ by glueing over $p$
is the subring of $B$ (containing $A$) given by
$$A':=\left\{b\in B \mid \exists \frac{a_{o}}{s_{o}}\in A_{p}\ \text{with}\ \overline{\frac{b}{1}}^{i}=\overline{\frac{a_{o}}{s_{o}}}^{i}\ \forall i\ \text{and, for}\  \frac{a}{s}\in A_{p},\ \overline{\frac{b}{1}}^{i}=\overline{\frac{a}{s}}^{i}\Leftrightarrow\overline{\frac{b}{1}}^{j}=\overline{\frac{a}{s}}^{j}\ \forall i,j\right\}.$$

Now, consider the following diagram
$$\xymatrix@R=1.5cm  @C=2.5cm{ A \ar[d]^{\mu}\ar[r]^{\iota} & \ar[d]^{\Phi}B\\ \frac{A_{p}}{pA_{p}} \ar[r]^{\Psi}&   D:=\frac{B_{p_{1}}}{p_{1}B_{p_{1}}}\times \dots  \times \frac{B_{p_{n}}}{p_{n}B_{p_{n}}}}$$
where $\iota$ is the natural embedding, $\mu(a)=\overline{\frac{a}{1}}\ \forall a\in A$, $\Phi(b)=(\overline{\frac{b}{1}}^{1}, ..., \overline{\frac{b}{1}}^{n})\ \forall b\in B$, and $\Psi(\overline{\frac{a}{s}})=(\overline{\frac{a}{s}}^{1}, ..., \overline{\frac{a}{s}}^{n})\ \forall \frac{a}{s}\in A_{p}$. Let $J:=\Ker(\Phi)$ and $J':=\Ker(\Psi)$ and note that $$p=\iota^{-1}(J)=\mu^{-1}(J').$$

\begin{corollary}\label{glu}
Under the above notation, the following assertions are equivalent:
\begin{enumerate}
\item $A'= A\bowtie^{\iota,\mu}(J,J')$;
\item For any $(\frac{a}{s},b)\in A_{p}\times B:\ a-sb\in\bigcap_{1\leq i\leq n}p_{i} \Rightarrow a-sa_{o}\in p$, for some $a_{o}\in A$.
\end{enumerate}
\end{corollary}

\begin{proof}
By \cite[Lemma 2]{Yan}, $A'$ can be identified with the pullback $\Phi\times_{D}\Psi$. Further, notice that $\Phi\circ\iota=\Psi\circ\mu$; i.e., the above diagram is commutative. Let $\pi: B\times \frac{A_{p}}{pA_{p}}\rightarrow B$ be the canonical projection and let $a\in A$. Then
$$\Psi\big(\overline{\frac{a}{1}}\big)=\big(\overline{\frac{a}{1}}^{1}, ..., \overline{\frac{a}{1}}^{n}\big)=\Phi(a)=\Phi\circ\pi\big(a,\overline{\frac{a}{1}}\big).$$
Hence $\Phi(A)\subseteq\Phi\circ\pi(\Phi\times_{D}\Psi)$. Therefore, by Proposition~\ref{cara}, (1) holds if and only if $\Phi\circ\pi(\Phi\times_{D}\Psi)\subseteq \Phi(A)$ if and only if for any $(\frac{a}{s},b)\in A_{p}\times B$,  $\overline{\frac{a}{s}}^{i}=\overline{\frac{b}{1}}^{i}\ \forall i$ forces $\overline{\frac{a}{s}}^{i}=\overline{\frac{a_{o}}{1}}^{i}\ \forall i$, for some $a_{o}\in A$ if and only if (2) holds.
\end{proof}

For example, if $A := \Z$ and $p := 2\Z$, then for any finite $\Z$-module $B$ (e.g., $\Z[i]$)
Condition (2) of Corollary~\ref{glu} always holds since, for any $n\in \Z$ and $s\in\Z \setminus 2\Z$, $n-sn\in 2\Z$.

\section{Basic algebraic properties of bi-amalgamations}

\noindent Throughout, let $f: A\rightarrow B$ and $g: A\rightarrow C$ be two ring homomorphisms and  $J,J'$ two ideals of $B$ and $C$, respectively, such that $I_{o}:=f^{-1}(J)=g^{-1}(J')$. Let
$$A\bowtie^{f,g}(J,J'):=\big\{(f(a)+j,g(a)+j') \mid a\in A, (j,j')\in J\times J'\big\}$$
be the bi-amalgamation of $A$ with $(B, C)$ along $(J, J')$ with respect to $(f,g)$.

This section studies basic algebraic properties of bi-amalgamations. Precisely, we investigate necessary and sufficient conditions for a bi-amalgamation to be a Noetherian ring, a domain, or a reduced ring. We will show that the transfer of these notions is made via the special rings $f(A)+J$ and $g(A)+J'$ (which correspond to $B$ and $C$, respectively, in the case when $f$ and $g$ are surjective).

We start with some basic ideal-theoretic properties of bi-amalgamations. For this purpose, notice first that $0\times J'$, $J\times 0$, and $J\times J'$ are particular ideals of $A\bowtie^{f,g}(J,J')$; and if $I$ is an ideal of $A$, then the set $$I\bowtie^{f,g}(J,J'):=\big\{(f(i)+j,g(i)+j')\mid i\in I, (j,j')\in J\times J'\big\}$$
is an ideal of $A\bowtie^{f,g}(J,J')$ containing $J\times J'$.

\begin{proposition}\label{prop cong}
Let $I$ be an ideal of $A$. We have the following canonical isomorphisms:
\begin{enumerate}
\item $\dfrac{A\bowtie^{f,g}(J,J')}{I\bowtie^{f,g}(J,J')}\cong\dfrac{A}{I+I_{o}}$.

\item $\displaystyle\frac{A\bowtie^{f,g}(J,J')}{0\times J'}\cong f(A)+J$ and $\displaystyle\frac{A\bowtie^{f,g}(J,J')}{J\times 0}\cong g(A)+J'$.

\item $\dfrac{A}{I_{o}}\cong\dfrac{A\bowtie^{f,g}(J,J')}{J\times J'}\cong \dfrac{f(A)+J}{J}\cong \dfrac{g(A)+J'}{J'}$.
\end{enumerate}
\end{proposition}

\begin{proof}
(1) Consider the mapping
$$\begin{array}{cccc}
\varphi:    & A     & \rightarrow   &   \displaystyle\frac{A\bowtie^{f,g}(J,J')}{I\bowtie^{f,g}(J,J')}\\
            & a     & \longmapsto       &   \overline{(f(a),g(a))}.
\end{array}$$
Clearly, $\varphi$ is a surjective ring homomorphism and one can check that $\Ker(\varphi)=I+I_{o}$.

(2) If $f(a)+j=0$ for some $a\in A$ and $j\in J$, then $g(a)+j'\in J'$ for any $j'\in J'$. So the kernel of the surjective canonical homomorphism $A\bowtie^{f,g}(J,J')\twoheadrightarrow f(A)+J$ coincides with $0\times J'$. Hence, the first isomorphism holds and the second one follows similarly.

(3) The first isomorphism is a particular case of (1) for $I=0$. Further, if $f(a)+j\in J$ for some $a\in A$ and $j\in J$, then $g(a)+j'\in J'$ for any $j'\in J'$. So the kernel of the canonical surjective homomorphism $$A\bowtie^{f,g}(J,J')\twoheadrightarrow\dfrac{f(A)+J}{J}$$ coincides with $J\times J'$.
\end{proof}

The fact that bi-amalgamations can be represented as pullbacks is an  important tool that one can use to investigate the algebraic properties of these constructions. The following results give examples of this use.

\begin{proposition}\label{noeth}
Under the above notation, we have:
\begin{center} $A\bowtie^{f,g}(J,J')$ is Noetherian  $\Leftrightarrow$ $f(A)+J$ and $g(A)+J'$ are Noetherian.\end{center}
\end{proposition}

\begin{proof}
In view of Proposition~\ref{prop cong}(2), we only need to prove the reverse implication. By Proposition~\ref{prop1},
$A\bowtie^{f,g}(J,J')=\alpha\times_{\frac{A}{I_{o}}}\beta$ determined by the ring homomorphisms $\alpha: f(A)+J\rightarrow A/I_{o}$, $f(a)+j\mapsto \bar{a}$ and $\beta: g(A)+J'\rightarrow A/I_{o}$, $g(a)+j'\mapsto \bar{a}$. Sine $f(A)+J$ is Noetherian, by  \cite[Proposition 4.10]{DFF1}, it suffices to show that $\Ker(\beta)=J'$ is a Noetherian module over $A\bowtie^{f,g}(J,J')$ with the module structure induced by the surjective canonical homomorphism $A\bowtie^{f,g}(J,J')\twoheadrightarrow g(A)+J'$. But, under this structure, $A\bowtie^{f,g}(J,J')$-submodules of $J'$ correspond to subideals of $J'$ in the Noetherian ring $g(A) + J'$. This leads to the conclusion.
\end{proof}

In view of Example~\ref{aa}, Proposition~\ref{noeth} recovers the special case of amalgamated algebras, as recorded in the next corollary.

\begin{corollary}[{\cite[Proposition 5.6]{DFF1}}]
Under the above notation, we have:
\begin{center} $A\bowtie^{f}J$ is Noetherian  $\Leftrightarrow$  $A$ and $f(A)+J$ are Noetherian.\end{center}
\qed\end{corollary}

As an illustrative example for Proposition~\ref{noeth} (of an original Noetherian ring which arises as a bi-amalgamation) is provided in Example~\ref{s4:exa1}.

Recall that the prime spectrum of a ring $R$ is said to be Noetherian if $R$ satisfies the ascending chain condition on radical ideals (or, equivalently, every prime ideal of $R$ is the radical of a finitely generated ideal) \cite{OP}. Let $\Spec(R)$ denote the prime spectrum of a ring $R$.

\begin{proposition}
Under the above notation, we have:
\begin{center} $\Spec\big(A\bowtie^{f,g}(J,J')\big)$ is Noetherian  $\Leftrightarrow$ $\Spec\big(f(A)+J\big)$ and $\Spec\big(g(A)+J'\big)$  are Noetherian.\end{center}
\end{proposition}

\begin{proof}
$A\bowtie^{f,g}(J,J')=\alpha\times_{\frac{A}{I_{o}}}\beta$ via the homomorphisms $\alpha: f(A)+J\rightarrow A/I_{o}$, $f(a)+j\mapsto \bar{a}$ and $\beta: g(A)+J'\rightarrow A/I_{o}$, $g(a)+j'\mapsto \bar{a}$. So, by \cite[Corollary 1.6]{F}, the prime spectra of $A\bowtie^{f,g}(J,J')$ and $A/I_{o}$ are Noetherian if and only if so are the prime spectra of $f(A)+J$ and $g(A)+J'$. But, by Proposition~\ref{prop cong}(3) if the prime spectrum of $A\bowtie^{f,g}(J,J')$ is Noetherian, then so is the spectrum of $A/I_{o}$ since this notion is stable under homomorphic image. This leads to the conclusion.
\end{proof}

The next result characterizes bi-amalgamations without zero divisors.

\begin{proposition}\label{domain}
Under the above notation, the following assertions are equivalent:
\begin{enumerate}
  \item $A\bowtie^{f,g}(J,J')$ is a domain;
  \item ``$J=0$ and $g(A)+J'$ is a domain" or ``$J'=0$ and $f(A)+J$ is a domain."
\end{enumerate}
\end{proposition}

\begin{proof}
Assume that $A\bowtie^{f,g}(J,J')$ is a domain. If $J\neq 0$ and $J'\neq 0$, then for nonzero elements $j\in J$ and $j'\in J'$ we have $(0,j')(j,0)=(0,0)$. Therefore, one of $J$ and $J'$ must be null; in such case,  $A\bowtie^{f,g}(J,J')$ collapse (up to an isomorphism) to $f(A)+J$ or $f(A)+J$ by Proposition~\ref{prop cong}(2). This leads to the conclusion.
\end{proof}

In view of Example~\ref{aa}, Proposition~\ref{domain} recovers the special case of amalgamated algebras, as recorded in the next corollary.

\begin{corollary}[{\cite[Proposition 5.2]{DFF1}}]
Under the above notation, assume $J\not=0$. Then:
\begin{center} $A\bowtie^{f}J$ is a domain  $\Leftrightarrow$  $f^{-1}(J)=0$ and $f(A)+J$ is a domain.\end{center}
\qed\end{corollary}

The next result characterizes bi-amalgamations without nilpotent elements.

\begin{proposition}\label{reduced}
Under the above notation, consider the following conditions:
\begin{enumerate}
\item[(a)] $f(A)+J$ is reduced and $J'\cap \nil(C)=0$,
\item[(b)] $g(A)+J'$ is reduced and $J\cap \nil(B)=0$,
\item[(c)] $A\bowtie^{f,g}(J,J')$ is reduced,
\item[(d)] $J\cap \nil(B)=0$ and $J'\cap \nil(C)=0$.
\end{enumerate}
Then:
\begin{enumerate}
\item $(a)\ \text{or}\ (b) \Rightarrow (c) \Rightarrow (d)$.
\item If $I_{o}$ is radical, then the four conditions are equivalent.
\item  If $f$ is surjective and $\Ker(f)\subseteq \Ker(g)$, then:
\begin{center} $A\bowtie^{f,g}(J,J')$ is reduced  $\Leftrightarrow$ $B$ is reduced and $J'\cap \nil(C)=0$.\end{center}
\end{enumerate}
\end{proposition}

\begin{proof}
(1) Let $(f(a)+j,g(a)+j')\in \nil\big(A\bowtie^{f,g}(J,J')\big)$. Then $f(a)+j\in \nil(f(A)+J)=0$. Hence, $a\in I_{o}$. Thus, $g(a)+j'\in J'\cap \nil(C)=0$. Consequently, $\nil\big(A\bowtie^{f,g}(J,J')\big)=0$. This proves $(a)\Rightarrow (c)$. Likewise for $(b)\Rightarrow (c)$.

Let $j\in \nil(B)\cap J$. Therefore, there is a positive integer $n$ such that $0=(j^n,0)=(j,0)^n$ in $A\bowtie^{f,g}(J,J')$. It follows that $j=0$ and hence $\nil(B)\cap J=0$. Similarly, $\nil(C)\cap J'=0$. This proves $(c)\Rightarrow (d)$.

(2) Next, assume that $I_{o}$ is radical, $J\cap \nil(B)=0$, and $J'\cap \nil(C)=0$. Let $f(a)+j\in \nil(f(A)+J)$. Then, there is a positive integer $n$ such that $(f(a)+j)^n=0$. Hence, $f(a)^n\in J$ and thus $a^n\in I_{o}$; that is, $a\in I_{o}$. So, $f(a)+j\in J\cap \nil(B)=0$, as desired. This proves $(d)\Rightarrow (a)$. Likewise for $(d)\Rightarrow (b)$.

(3) In view of (1), it suffices to observe that  $f(a^{n})=0$, for some positive integer, forces $(f(a),g(a))^{n}=0$, yielding $f(a)=0$.
\end{proof}

\begin{remark}\label{Bnil}
If $f(A)+J$ and $g(A)+J'$ are both reduced, then $A\bowtie^{f,g}(J,J')$ is reduced by Proposition~\ref{reduced}. The converse is not true in general. A counter-example (for the special case of amalgamated algebras) is given in \cite[Remark 5.5 (3)]{DFF1}.
\end{remark}

In view of Example~\ref{aa}, Proposition~\ref{reduced} recovers the special case of amalgamated algebras, as recorded in the next corollary.

\begin{corollary}[{\cite[Proposition 5.4]{DFF1}}]
Under the above notation, we have:
\begin{center} $A\bowtie^{f}J$ is reduced  $\Leftrightarrow$  $A$ is reduced and $J\cap \nil(B)=0$.\end{center}
\qed\end{corollary}

As an illustrative example for Propositions \ref{noeth} \& \ref{domain} \& \ref{reduced}, we provide an original reduced Noetherian ring with zero divisors which arises as a bi-amalgamation.

\begin{example}\label{s4:exa1}
Consider the surjective ring homomorphism $f: \Z[X]\twoheadrightarrow \Z[\sqrt{2}]$, $p(X)\mapsto p(\sqrt{2})$ and the principal ideal $J:=(\sqrt{2})$ of $\Z[\sqrt{2}]$. Let $p\in \Z[X]$ and write it as $p=(X^2-2)q(X)+aX+b$ for some $a,b\in\Z$ and $q\in \Z[X]$. Then, one can verify that $p(\sqrt{2})\in J$ if and only if $b\in 2\Z$. That is,
$$I_{o}:=f^{-1}(J)=\big\{p\in \Z[X] \mid p(0)\in 2\Z\big\}.$$
Now, consider the ring homomorphism $\alpha: \Z[\sqrt{2}]\twoheadrightarrow \dfrac{\Z[X]}{I_{o}}$, $a+b\sqrt{2}\mapsto \bar{a}$. It follows, by Proposition~\ref{prop1} and Propositions \ref{noeth} \& \ref{domain} \& \ref{reduced}, that
$$\Z[X]\bowtie^{f,f}(J,J)=\alpha\times_{\frac{\Z[X]}{I_{o}}}\alpha=\big\{(a+b\sqrt{2}, c+d\sqrt{2})\mid a,b,c,d\in \Z,\ a-c\in 2\Z\big\}$$
is a reduced Noetherian ring that is not a domain (since $\Z[\sqrt{2}]$ is a Noetherian domain and $J\not=0$).
\end{example}

\section{The prime ideal structure of bi-amalgamations}

\noindent Throughout, let $f: A\rightarrow B$ and $g: A\rightarrow C$ be two ring homomorphisms and  $J,J'$ two ideals of $B$ and $C$, respectively, such that $I_{o}:=f^{-1}(J)=g^{-1}(J')$. Let
$$A\bowtie^{f,g}(J,J'):=\big\{(f(a)+j,g(a)+j') \mid a\in A, (j,j')\in J\times J'\big\}$$
be the bi-amalgamation of $A$ with $(B, C)$ along $(J, J')$ with respect to $(f,g)$.

This section investigates the prime ideal structure of bi-amalgamations and their localizations at prime ideals. We also establish necessary and sufficient conditions for a bi-amalgamation to be local.

Next, we describe the prime (and maximal) ideals of bi-amalgamations. To this purpose, let's adopt the following notation:
$$\begin{array}{lcl}
Y           &:=     &\Spec(f(A)+J)\\
Y'          &:=     &\Spec(g(A)+J')
\end{array}$$
and, for $L\in Y$ and $L'\in Y'$, consider the prime ideals of $A\bowtie^{f,g}(J,J')$ given by:
$$\begin{array}{rcl}
\bar{L}     &:=     &\big(L\times (g(A)+J')\big)\cap\big(A\bowtie^{f,g}(J,J')\big)\\
            &=      &\big\{(f(a)+j,g(a)+j')\mid a\in A, (j,j')\in J\times J', f(a)+j\in L\big\},\\
\bar{L'}    &:=     &\big((f(A)+J)\times L'\big)\cap\big(A\bowtie^{f,g}(J,J')\big)\\
            &=      &\big\{(f(a)+j,g(a)+j')\mid a\in A, (j,j')\in J\times J', g(a)+j'\in L'\big\}.
\end{array}$$

The next two lemmas are needed for the proof of Proposition~\ref{spec}. Recall that if $I$ is an ideal of $A$, then $$I\bowtie^{f,g}(J,J'):=\big\{(f(i)+j,g(i)+j')\mid i\in I, (j,j')\in J\times J'\big\}$$
is an ideal of $A\bowtie^{f,g}(J,J')$. As an immediate consequence of Proposition~\ref{prop cong}(1), we have the following lemma.

\begin{lemma}\label{s5:lem1}
Let $I$ be an ideal of $A$. Then, $I\bowtie^{f,g}(J,J')$ is a prime (resp., maximal) ideal of $A\bowtie^{f,g}(J,J')$ if and only if $I+I_{o}$ is a prime (resp., maximal) ideal of $A$.
\qed\end{lemma}

An element of $Y$ (resp., $Y'$) containing $J$ (resp., $J'$) has a special form, as shown by the next lemma.

\begin{lemma}\label{s5:lem2}
Let $L\in Y$ (resp., $Y'$) containing $J$ (resp., $J'$). Then:
$$\bar{L}=f^{-1}(L)\bowtie^{f,g}(J,J')\ \big(\text{resp.,}\ =g^{-1}(L)\bowtie^{f,g}(J,J')\big).$$
\end{lemma}

\begin{proof}
Let $L\in Y$ containing $J$. Notice first that $f^{-1}(L)$ is a prime ideal of $A$ containing $I_{o}:=f^{-1}(J)$ so that $f^{-1}(L)\bowtie^{f,g}(J,J')$ is a prime ideal of $A\bowtie^{f,g}(J,J')$ by Lemma~\ref{s5:lem1}. Moreover,  for any $a\in A$ and $j\in J$, one can easily see that $f(a)+j\in L$ if and only if $a\in f^{-1}(L)$. Thus, $\overline{L}=f^{-1}(L)\bowtie^{f,g}(J,J')$. Likewise for $L\in Y'$.
\end{proof}

\begin{proposition}\label{spec}
Under the above notation, let $P$ be a prime ideal of $A\bowtie^{f,g}(J,J')$. Then
\begin{enumerate}
\item $J\times J'\subseteq P$ $\Leftrightarrow$ $\exists!\ p\supseteq I_{o}$ in $\Spec(A)$ such that $P=p\bowtie^{f,g}(J,J').$\\
In this case, $\exists\ L\supseteq J$ in $Y$ and $\exists\ L'\supseteq J'$ in $Y'$ such that $P=\bar{L}=\bar{L'}.$
\vspace{.2cm}

\item $J\times J'\nsubseteq P$ $\Leftrightarrow$  $\exists!\ L\in Y$ (or $Y'$) such that  $J\nsubseteq L$ (or $J'\nsubseteq L$) and $P=\bar{L}$.\\
In this case, $(A\bowtie^{f,g}(J,J'))_{P}\cong (f(A)+J)_L\ \big(\text{ or }\ (A\bowtie^{f,g}(J,J'))_{P}\cong (g(A)+J')_L\big).$
\vspace{.2cm}

\noindent Consequently, we have
$$\Spec\big(A\bowtie^{f,g}(J,J')\big)=\big\{\bar{L}\mid L\in \Spec\big(f(A)+J\big)\cup \Spec\big(g(A)+J'\big)\big\}.$$
\end{enumerate}

\end{proposition}

\begin{proof}
(1) We only need to prove ($\Rightarrow$). Assume $J\times J'\subseteq P$ and consider the ideal $p$ of $A$ given by
$$p:=\big\{a\in A\mid \exists\ (j,j')\in J\times J'\ \text{such that}\ (f(a)+j,g(a)+j')\in  P\big\}.$$
Clearly, the fact $J\times J'\subseteq P$ forces $I_{o}\subseteq p$. Moreover, we have $P\subseteq p\bowtie^{f,g}(J,J')$. For the reverse inclusion, let $a\in p$. So there exists $(j_{1},j'_{1})\in J\times J'$ such that $(f(a)+j_{1},g(a)+j'_{1})\in  P$. Hence, for every $(j,j')\in J\times J'$, we obtain $$(f(a)+j,g(a)+j')=(f(a)+j_{1},g(a)+j'_{1})+(j-j_{1},j'-j'_{1})\in P$$
since $J\times J'\subseteq P$. It follows that $$P= p\bowtie^{f,g}(J,J').$$
By Lemma~\ref{s5:lem1}, $p$ is a prime ideal of $A$. By Proposition~\ref{prop cong}(1), $p$ must be unique since it contains $I_{o}$.

Next,  let $L:=f(p)+J$. One can verify that $L$ is a prime ideal of $f(A)+J$ with $p\subseteq f^{-1}(L)$. Now, let $a\in f^{-1}(L)$. Then $f(a)=f(x)+j$ for some $x\in p$ and $j\in J$. Hence $(a-x)\in I_{o}\subseteq p$, whence $a\in p$. So, $$f^{-1}(L)= p.$$ It follows, via Lemma~\ref{s5:lem2}, that $$\bar{L}=f^{-1}(L)\bowtie^{f,g}(J,J')=p\bowtie^{f,g}(J,J')=P.$$
Note that for $L':=g(p)+J'$, the same arguments lead to
$$P=\bar{L}=\bar{L'}.$$

(2)  We only need to prove ($\Rightarrow$). Assume $J\times J'\nsubseteq P$. By Proposition \ref{conductor} and \cite[Lemma 1.1.4(3)]{FHP}, there is a unique prime  $Q$ of $(f(A)+J)\times (g(A)+J')$ such that
$$P=Q\cap A\bowtie^{f,g}(J,J')\ \text{ with}\ \big((f(A)+J)\times (g(A)+J')\big)_{Q}=\big(A\bowtie^{f,g}(J,J')\big)_{P}.$$
Then either $Q=L\times (g(A)+J')$ for some prime ideal $L\in Y$ or $Q=(f(A)+J)\times L'$ for some prime ideal $L'\in Y'$. That is,
$$P=\bar{L}\ \text{ or }\ P=\bar{L'}.$$
Accordingly, we'll have
$$(A\bowtie^{f,g}(J,J'))_{P}\cong (f(A)+J)_L\ \text{ or }\ (A\bowtie^{f,g}(J,J'))_{P}\cong (g(A)+J')_{L'}$$
completing the proof of the proposition.
\end{proof}

Next, as an application of Proposition~\ref{spec}, we establish necessary and sufficient conditions for a bi-amalgamation to be local. Notice at this point that, in the presence of the equality $f^{-1}(J)=g^{-1}(J')$, $J\not=B$ if and only if $J'\not=C$.

\begin{proposition}\label{local}
Under the above notation, we have
\begin{enumerate}
\item $A\bowtie^{f,g}(J,J')$ is local $\Leftrightarrow$  $J\not= B$ and $f(A)+J$ \& $g(A)+J'$ are local.\\
Moreover, the maximal ideal of $A\bowtie^{f,g}(J,J')$ has the form $\m\bowtie^{f,g}(J,J')$, where $\m$ is the unique maximal ideal of $A$ containing $I_{o}$.
\bigskip

\item Suppose that $A$ is local. Then: $$A\bowtie^{f,g}(J,J')\ \text{is local}\ \Leftrightarrow J\times J'\subseteq Jac(B\times C).$$
\end{enumerate}
\end{proposition}

\begin{proof}
(1) Notice first that if $J=B$, (hence $J'=C$ and) then $A\bowtie^{f,g}(J,J')= B\times C$ which is never local. Assume that $A\bowtie^{f,g}(J,J')$ is local. Then $J\not= B$ and, by Proposition~\ref{prop cong}(2), both  $f(A)+J$ and $g(A)+J'$ are local. Moreover, $I_{o}\not=A$. Therefore, there is $\m\supseteq I_{o}$ maximal in $A$. By Lemma~\ref{s5:lem1}, $\m\bowtie^{f,g}(J,J')$ is the maximal ideal of $A\bowtie^{f,g}(J,J')$. Then, the uniqueness of $\m$ is ensured by Proposition~\ref{prop cong}(1).

Next assume that $J\not= B$ and $f(A)+J$ \& $g(A)+J'$ are local. Let $M$ be a maximal ideal of $A\bowtie^{f,g}(J,J')$. We claim that $J\times J'\subseteq M$. Deny. Then, by Proposition~\ref{spec}(2), there is a unique prime $L$, say, of $f(A)+J$ such that $M=\bar{L}$ and $J\nsubseteq L$. Further, the uniqueness of $L$ and maximality of $M$ force $L$ to be a (in fact, the) maximal ideal of $f(A)+J$. It follows that $J\subseteq L$ (since $J\not= B$), the desired contradiction. Therefore, $$J\times J'\subseteq M.$$ So, by Proposition \ref{spec}(1),  there is a (unique) prime ideal $\m$ of $A$ containing $I_{o}$  such that
$$M=\m\bowtie^{f,g}(J,J').$$
By Lemma~\ref{s5:lem1}, $\m$ is maximal in $A$. By Proposition~\ref{prop cong}(3), $\dfrac{A}{I_{o}}\cong\dfrac{f(A)+J}{J}$ is local with maximal ideal $\dfrac{\m}{I_{o}}$. This forces $M$ to be the unique maximal ideal of $A\bowtie^{f,g}(J,J')$.

(2) ($\Rightarrow$) In this direction we don't need the assumption ``$A$ is local." Assume that $A\bowtie^{f,g}(J,J')$ is local. By (1), necessarily, its maximal ideal contains $J\times J'$. Let $(j,j')\in J\times J'$ and $(b,c)\in B\times C$. Then, $(b,c)(j,j')\in J\times J'$. Thus, $(1,1)-(b,c)(j,j')$ is invertible in $A\bowtie^{f,g}(J,J')$ (and so in $B\times C$). Hence, $J\times J'\subseteq Jac(B\times C)$.

($\Leftarrow$) Assume that $A$ is local and $J\times J'\subseteq Jac(B\times C)$. Let $a$ be a unit of $A$. We claim that $(f(a)+j,g(a)+j')$ is a unit of  $A\bowtie^{f,g}(J,J')$ for every $(j,j')\in J\times J'$. Indeed, $f(a)+j$ and $g(a)+j'$ are, respectively,  units in $B$ and $C$ since $J\times J'\subseteq Jac(B\times C)$. Thus, there exist $u\in B$ and $v\in C$ such that $(f(a)+j)u=1$ and $(g(a)+j')v=1$. Hence,
$$(f(a)+j,g(a)+j')(f(a^{-1})-uf(a^{-1})j,g(a^{-1})-vg(a^{-1})j')=(1,1);$$
that is, $(f(a)+j,g(a)+j')$ is a unit of $A\bowtie^{f,g}(J,J')$. Next, let $(f(a)+j_1,g(a)+j'_1)$ be a nonunit element of $A\bowtie^{f,g}(J,J')$. So, $a$ is a nonunit of $A$. Moreover, for any $(f(b)+j_2,g(b)+j_2')\in A\bowtie^{f,g}(J,J')$, we have $$(1,1)-(f(b)+j_2,g(b)+j'_2)(f(a)+j_1,g(a)+j'_1)=(f(1-ba)+j_3,g(1-ba)+j'_3)$$
for some $j_{3}\in J$ and $j'_{3}\in J'$. Further, $1-ba$ is a unit of $A$ since $A$ is local. Hence, $(1,1)-(f(b)+j_2,g(b)+j'_2)(f(a)+j_1,g(a)+j'_1)$ is a unit of $A\bowtie^{f,g}(J,J')$. This proves that $A\bowtie^{f,g}(J,J')$ is local.
\end{proof}

In view of Example~\ref{aa}, Proposition~\ref{local} recovers the special case of amalgamated algebras and amalgamated duplications, as recorded in the next corollaries.

\begin{corollary}
Under the above notation, the following assertions are equivalent:
\begin{enumerate}
\item $A\bowtie^fJ$ is local;
\item $J\not=B$ and $A$ \& $f(A)+J$ are local;
\item $A$ is local and $J\subseteq Jac(B)$.

\end{enumerate}
\qed\end{corollary}

\begin{corollary}[{\cite[Corollary 6]{D} \& \cite[Theorem 3.5(1.e)]{DF1} \& \cite[Proposition 2.2]{DF2}}]
Let $A$ be a ring and $I$ a proper ideal of $A$. Then, $A\bowtie I$ is local if and only if $A$ is local.
\qed\end{corollary}

Next, we describe the localizations of $A\bowtie^{f,g}(J,J')$ at its prime ideals which contain $J\times J'$. Recall that, given a ring $R$, an ideal $I$ of $R$, and $S$ a multiplicatively closed subset of $R$ with $S\cap I =\emptyset$, then $S+I$  is a multiplicatively closed subset of $R$.

\begin{proposition}\label{localization}
Let $p$ be a prime ideal of $A$ containing $I_{o}$ and let $P:=p\bowtie^{f,g}(J,J')$. Consider the multiplicative subsets $S:=f(A-p)+J$ of $B$ and $S':=g(A-p)+J'$ of $C$. Let $f_p: A_p\rightarrow B_S$ and $g_p: A_p\rightarrow C_{S'}$ be the ring homomorphisms induced by $f$ and $g$. Then: $$f_p^{-1}(J_S)=g_p^{-1}(J'_{S'})=(I_{o})_p$$ and  $$\big(A\bowtie^{f,g}(J,J')\big)_{P}\cong A_p\bowtie^{f_p,g_p}(J_S,J'_{S'}).$$
\end{proposition}

\begin{proof}
It is easy to show that $f_p^{-1}(J_S)=g_p^{-1}(J'_{S'})=(I_{o})_p$. Moreover, by Proposition \ref{prop1}, $A_p\bowtie^{f_p,g_p}(J_S,J'_{S'})$ is the fiber product of $\alpha: f_p(A_p)+J_S\rightarrow A_p/(I_{o})_p$ and $\beta: g_p(A_p)+J'_{S'}\rightarrow A_p/(I_{o})_p$. On the other hand, $\pi_B(A\bowtie^{f,g}(J,J')-P)=S$ and $\pi_C(A\bowtie^{f,g}(J,J')-P)=S'$. Then, the fact that $(A\bowtie^{f,g}(J,J'))_{P}$ is isomorphic to $A_p\bowtie^{f_p,g_p}(J_S,J'_{S'})$ follows from \cite[Proposition 1.9]{F}.
\end{proof}

\begin{remark}
If $P$ is a prime ideal of $A\bowtie^{f,g}(J,J')$ which contains $J\times J'$, then by Proposition \ref{spec},
there exists a (unique) prime ideal $p$ (which contains $I_{o}$) such that $P=p\bowtie^{f,g}(J,J')$. Thus, by Proposition \ref{conductor}  and Proposition \ref{localization}, one can obtain
a conductor square of the form:
$$\xymatrix@R=2cm  @C=2cm{ (A\bowtie^{f,g}(J,J'))_{P} \ar@{>>}[d]^{\mu_2}\ar[r]^{\iota_2} & \ar@{>>}[d]^{\mu_1}(f_p(A_p)+J_S)\times (g_p(A_p)+J'_{S'})\\
\displaystyle\frac{A_{p}}{I_{o}A_{p}}\quad \quad \ar[r]^{\iota_1}&  \quad \quad \displaystyle \frac{A_{p}}{I_{o}A_{p}}\times \frac{A_{p}}{I_{o}A_{p}}}$$
\end{remark}



\begin{thebibliography}{99}

\par\bibitem{An}    D. D. Anderson, Commutative rings, in: J. Brewer, S. Glaz, W. Heinzer, B. Olberding (Eds.), Multiplicative Ideal Theory in Commutative Algebra, Springer, New York, 2006, pp. 1--20.
\par\bibitem{BSh2}   M. Boisen Jr. and P. Sheldon, CPI-extension: overrings of integral domains with special prime spectrum, Canad. J. Math. 29 (1977) 722--737.
\par\bibitem{CJKM}  M. Chhiti, M. Jarrar, S. Kabbaj, and N. Mahdou, Pr\"ufer conditions in an amalgamated duplication of a ring along an ideal, Comm. Algebra, to appear
\par\bibitem{D}     M. D'Anna, A construction of Gorenstein rings, J. Algebra  306 (6)  (2006) 507--519.
\par\bibitem{DFF1}  M. D'Anna, C. Finocchiaro and M. Fontana, Amalgamated algebras along an ideal, in: M. Fontana, S. Kabbaj, B. Olberding, I. Swanson (Eds.), Commutative Algebra and its Applications, Walter de Gruyter, Berlin, 2009, pp. 155--172.
\par\bibitem{DFF2}  M. D'Anna, C. Finocchiaro and M. Fontana, Properties of chains of prime ideals in an amalgamated algebra along an ideal, J. Pure Appl. Algebra  214 (9) (2010) 1633--1641.
\par\bibitem{DF1}   M. D'Anna and M. Fontana, An amalgamated duplication of a ring along an ideal: the basic properties, J. Algebra Appl.  6 (3) (2007) 443--459.		
\par\bibitem{DF2}   M. D'Anna and M. Fontana, The amalgamated duplication of a ring along a multiplicative-canonical ideal, Ark. Mat.  45 (2) (2007) 241--252.
\par\bibitem{Do}    J. L. Dorroh, Concerning adjunctions to algebras, Bull. Amer. Math. Soc. 38 (1932), 85--88.
\par\bibitem{F}     M. Fontana, Topologically defined classes of commutative rings, Ann. Mat. Pura Appl.  123 (1980), 331--355.
\par\bibitem{FHP}   M. Fontana, J. A. Huckaba and  I. J. Papick, Pr\"ufer Domains, Marcel Dekker, New York, 1997.
\par\bibitem{Hu}    J. A. Huckaba, Commutative Rings with Zero Divisors, Marcel Dekker, New York, 1988.
\par\bibitem{MY}    H. Maimani and S. Yassemi, Zero-divisor graphs of amalgamated duplication of a ring along an ideal, J. Pure Appl. Algebra  212 (1) (2008) 168--174.
\par\bibitem{N}     M. Nagata, Local Rings, Interscience, New York 1962.
\par\bibitem{OP}     J. Ohm and R. L. Pendleton, Rings with Noetherian spectrum, Duke Math. J.  35 (1968) 631--639.
\par\bibitem{Ped}    C. Pedrini, Incollamenti di ideali primi e gruppi di Picard, Rend. Sem. Math. Univ. Padova 48 (1973) 39--66.
\par\bibitem{Sh}    J. Shapiro, On a construction of Gorenstein rings proposed by M. D'Anna, J. Algebra  323 (4)  (2010) 1155--1158.
\par\bibitem{Tam}    G. Tamone, Sugli incollamenti di ideali primi, Bollettino U.M.I. 14 (1977) 810--825.
\par\bibitem{Tra}    C. Traverso, Seminormality and Picard group, Ann. Sc. Norm. Sup. Pisa 24 (1970) 585--595.
\par\bibitem{Yan}    H. Yanagihara, On glueings of prime ideals, Hiroshima Math. J. 10 (2) (1980) 351--363.
\end{thebibliography}
\end{document}